\documentclass[12pt]{article}

\usepackage{amsmath}
\usepackage{mathtools}
\usepackage{amssymb}
\usepackage{here}
\usepackage{mathrsfs}
\usepackage{enumerate}
\usepackage{multirow}
\usepackage{color}
\usepackage{fullpage}
\usepackage{dsfont}
\usepackage{cite}

\definecolor{lblue}{RGB}{0,110,152}

\definecolor{dred}{RGB}{171,67,53}

\newtheorem{theorem}{Theorem}%[\arabic{section}]
\newtheorem{corollary}[theorem]{Corollary}
\newtheorem{proposition}[theorem]{Proposition}
\newtheorem{lemma}[theorem]{Lemma}
\newtheorem{define}[theorem]{Definition}

\DeclareMathOperator*{\col}{col}

\DeclareMathOperator*{\diag}{diag}
\DeclareMathOperator*{\rank}{rank}

\DeclareMathOperator{\E}{ \mathbb{E}}

\def\E{\mathbb{E}}

\providecommand\X[1]{\boldsymbol{X_{#1}}}
\providecommand\Z[1]{\boldsymbol{Z_{#1}}}

\providecommand\phib{\boldsymbol{\emptyset}}

\def\llongrightarrow{\relbar\joinrel\relbar\joinrel\relbar\joinrel\rightarrow}

\providecommand{\rarrow}[1]{\stackrel{#1}{\llongrightarrow}}

\def\Xz{\boldsymbol{X}}
\def\Zz{\boldsymbol{Z}}

\providecommand\X[1]{\boldsymbol{X_{#1}}}
\providecommand\Z[1]{\boldsymbol{Z_{#1}}}
\providecommand\phib{\boldsymbol{\emptyset}}

\newenvironment{proof}{{\it Proof :~}}{\hfill$\diamondsuit$\\}

\begin{document}

%\setstretch{0.7}

%\begin{frontmatter}

\title{Robust and structural ergodicity analysis of stochastic biomolecular networks involving synthetic antithetic integral controllers\thanks{This paper is the expanded version of the paper of the same name that will appear in the proceedings of the 20th IFAC World Congress.}}

\author{Corentin Briat and Mustafa Khammash\thanks{Corentin Briat and Mustafa Khammash are with the Department of Biosystems Science and Engineering, ETH-Z\"{u}rich, Switzerland; email: mustafa.khammash@bsse.ethz.ch, corentin.briat@bsse.ethz.ch, corentin@briat.info; url: https://www.bsse.ethz.ch/ctsb/, http://www.briat.info.}}

\date{}

\maketitle

%\begin{keyword}
%Stochastic reaction networks; antithetic integral control; synthetic biology; robustness; polynomial methods
%\end{keyword}

\begin{abstract}
\noindent The concepts of ergodicity and output controllability have been shown to be fundamental for the analysis and synthetic design of closed-loop stochastic reaction networks, as exemplified by the use of antithetic integral feedback controllers. In [Gupta, Briat \& Khammash, PLoS Comput. Biol., 2014], some ergodicity and output controllability conditions for unimolecular and certain classes of bimolecular reaction networks were obtained and formulated through linear programs. To account for context dependence, these conditions were later extended in [Briat \& Khammash, CDC, 2016] to reaction networks with uncertain rate parameters using simple and tractable, yet potentially conservative, methods. Here we develop some exact theoretical methods for verifying, in a robust setting, the original ergodicity and output controllability conditions based on algebraic and polynomial techniques. Some examples are given for illustration.
%
%Accuracy is then achieved here at the expense of an increase of the computational complexity of the computational problems which are NP-hard by nature. When possible the complexity is reduced by exploiting or making relevant assumptions on the structure of the problem. Conservative, yet simple and useful, sufficient conditions are also provided for completeness. Some examples are given to illustrate these results.
\end{abstract}
%\end{frontmatter}

\section{Introduction}

The main objective of synthetic biology is the rational and systematic design of biological networks that can achieve de-novo functions such as the heterologous production of a metabolite of interest \cite{Ro:06}. Besides the obvious necessity of developing experimental methodologies allowing for the reliable implementation of synthetic networks, tailored theoretical and computational tools for their design, their analysis and their simulation also need to developed. Indeed, theoretical tools that could predict certain properties (e.g. a stable/oscillatory/switching behavior, controllable trajectories, etc.) of a synthetic biological network from an associated model formulated, for instance, in terms of a reaction network \cite{Feinberg:72,Horn:72,Goutsias:13}, could pave the way to the development of iterative procedures for the systematic design of efficient synthetic biological networks. Such an approach would allow for a faster design procedure than those involving fastidious experimental steps, and would give insights on how to adapt the current design in order to improve a certain design criterion. This way, synthetic biology would become conceptually much closer to existing theoretically-driven engineering disciplines, such as control engineering.
\begin{figure}[h]
%\vspace{0pt}
%\hspace{15mm}
%  \includegraphics[width=0.4\textwidth]{./Figures/Workflow.pdf}
%  \caption{Synthetic circuit design workflow involving an analysis step in order to refine the proposed designs.}\label{fig}
\centering
  \includegraphics[width=0.8\textwidth]{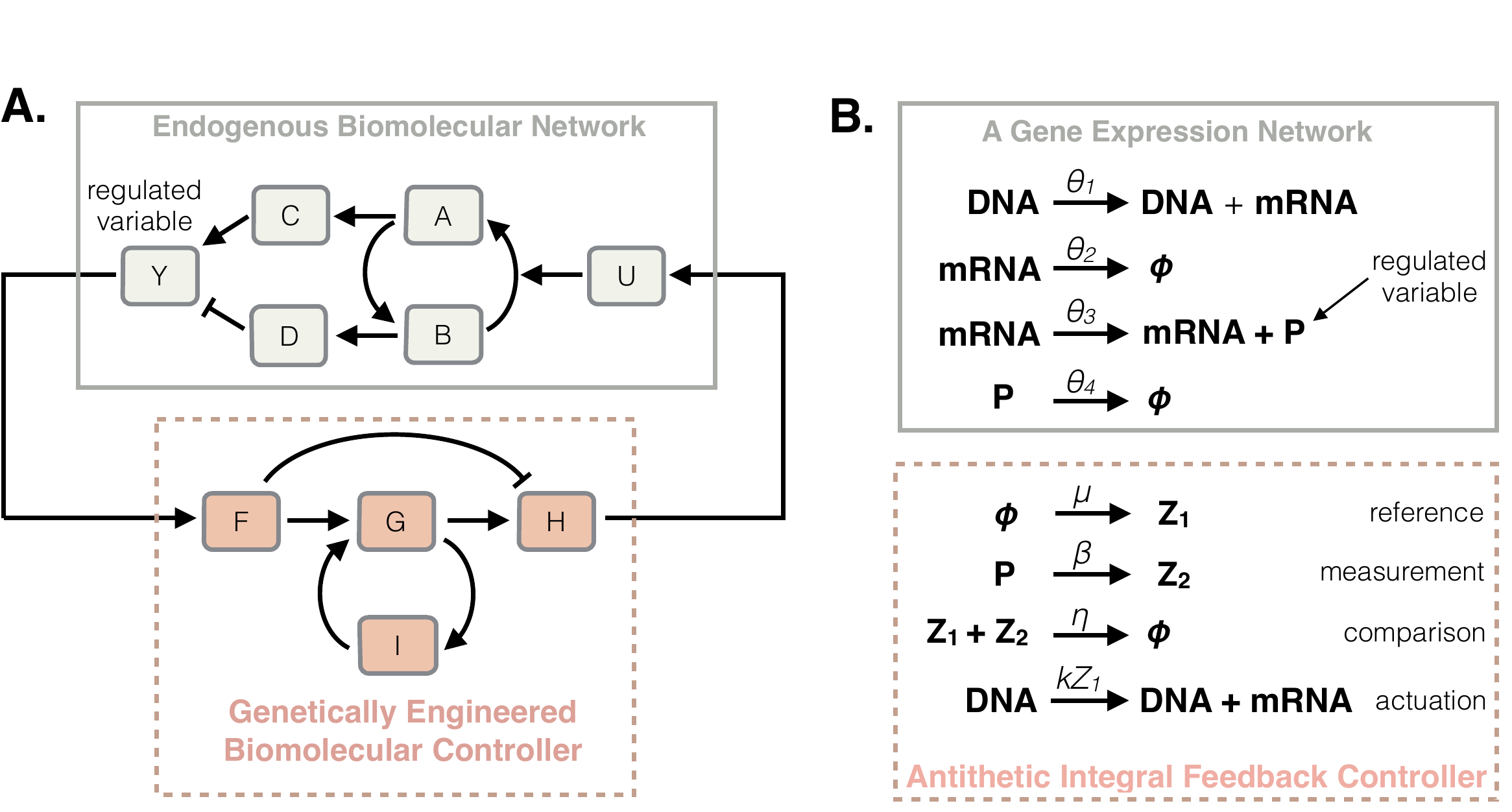}
  \caption{\textbf{A.} A synthetic feedback loop involving an endogenous network controlled with a synthetic feedback controller. \textbf{B.} A gene expression network (top) and an antithetic integral controller (bottom) as examples of endogenous and synthetic networks.}\label{fig}
\end{figure}
However, while such methods are well-developed for deterministic models (i.e. deterministic reaction networks), they still lag behind in the stochastic setting. This lack of tools is quite problematic since it is now well-known that stochastic reaction networks \cite{Goutsias:13} are versatile modeling tools that can capture the inherent stochastic behavior of living cells \cite{Thattai:01,Elowitz:02} and can exhibit several interesting properties that are absent for their deterministic counterparts \cite{Vilar:02,Paulsson:00,Briat:15e,Gupta:16}. Under the well-mixed assumption, it is known \cite{Anderson:11,Anderson:15} that such random dynamics can be well represented as a continuous-time jump Markov process evolving on the $d$-dimensional nonnegative integer lattice where $d$ is the number of distinct molecular species involved in the network. Sufficient conditions for checking the ergodicity of open unimolecular and bimolecular stochastic reaction networks has have been proposed in \cite{Briat:13i} and formulated in terms of linear programs. The concept of ergodicity is of fundamental importance as it can serve as a basis for the development of a control theory for biological systems. Indeed, verifying the ergodicity of a control system, consisting for instance of an endogenous biomolecular network controlled by a synthetic controller (see Fig.~\ref{fig}\textbf{A}), would prove that the closed-loop system is well-behaved (e.g. ergodic with bounded first- and second-order moments) and that the designed control system achieves its goal (e.g. set-point tracking and perfect adaptation properties). This  procedure is analogous to that of checking the stability of a closed-loop system in the deterministic setting; see e.g. \cite{DelVecchio:16}. Additionally, designing synthetic circuits  achieving a given function that are provably ergodic could allow for the rational design of synthetic networks that can exploit noise in their function. A recent example is that of the antithetic integral feedback controller proposed in \cite{Briat:15e} (see also Fig.~\ref{fig}\textbf{B}) that has been shown to induce an ergodic closed-loop network when some conditions on the endogenous network to be controlled are met.

A major limitation of the ergodicity conditions obtained in \cite{Briat:13i,Briat:15e} is that they only apply to networks with fixed and known rate parameters -- an assumption that is rarely met in practice as the rate parameters are usually poorly known and context dependent. This has motivated the consideration of networks with uncertain rate parameters in \cite{Briat:16cdc} using two different approaches. The first one is quantitative and considers networks having an interval matrix as characteristic matrix \cite{Moore:09}. It was notably shown that checking the ergodicity and the output controllability of those networks reduces to checking the Hurwitz stability of a single matrix and the output controllability of a single positive linear system. It was also shown that these conditions exactly write as a simple linear program having the same complexity as the program associated with the nominal case; i.e. in the case of constant and fixed rate parameters. The second approach is qualitative and is based on the theory of sign-matrices \cite{Jeffries:77,Brualdi:95} which has been extensively studied and considered for the qualitative analysis of dynamical systems. Sign-matrices have also been considered in the context of reaction networks, albeit much more sporadically; see e.g. \cite{Helton:09,Helton:10,Briat:14c,Briat:16cdc,Giordano:16}. In this case, again, the conditions obtained in \cite{Briat:14c,Briat:16cdc} can be stated as a very simple linear program that can be shown to be equivalent to some graph theoretical conditions.

However, these approaches can be very conservative when the entries of the characteristic matrix of the network are not independent -- a situation that appears when conversion reactions are involved in the networks. In order to solve this problem, the actual parameter dependence needs to be exactly captured and many approaches exist to attack this problem such as, to cite a few, $\mu$-analysis \cite{Doyle:82a,Packard:93}, small-gain methods \cite{Packard:93,Khammash:93}, eigenvalue and perturbation methods \cite{Michiels:07bk,Seyranian:03}, interval matrices \cite{Moore:09,Briat:16cdc}, sign-matrices \cite{Helton:09,Helton:10,Briat:14c,Briat:16cdc,Giordano:16}, Lyapunov methods \cite{Briat:book1,Blanchini:99,Angeli:09,Blanchini:14}, etc.

As some of the methods just cited above can be conservative or may yield too complex conditions, we propose to develop an approach that is tailored to our problem by exploiting its inherent properties. Several conditions for the robust ergodicity of unimolecular and biomolecular networks are first obtained in terms of a sign switching property for the determinant of the upper-bound of the characteristic matrix of the network. This condition also alternatively formulates as the existence of a positive vector depending polynomially on the uncertain parameters and satisfying certain inequality conditions. The complexity of the problem is notably reduced by exploiting the Metzler structure\footnote{A matrix is Metzler if its off-diagonal elements are nonnegative.} of the matrices involved and through the use of various algebraic results such as the Perron-Frobenius theorem. The structural ergodicity of unimolecular networks is also considered and shown to reduce to the analysis of constant matrices when some certain realistic assumptions are met. It is notably shown in the examples that this latter result can be applied to bimolecular networks in some situations. The examples also illustrate that the proposed approach can be used to establish the robust or structural ergodicity of reaction networks for which the methods proposed in \cite{Briat:16cdc} fail. %Note, finally, that the robust/structural output controllability problem is not considered here for the simple reason that it was fully solved in \cite{Briat:16cdc}.

\textbf{Outline.} Preliminaries on reaction networks, ergodicity analysis and antithetic integral control are given in Section \ref{sec:prel}. Section \ref{sec:rob_ergdocity} is devoted to the robust ergodicity analysis of unimolecular and bimolecular reactions networks while the problem of establishing the structural ergodicity of unimolecular reaction networks is addressed in Section \ref{sec:struct_ergdocity}. Examples are finally treated in Section \ref{sec:examples}.

\textbf{Notations.} The standard basis for $\mathbb{R}^d$ is denoted by $\{e_i\}_{i=1}^d$. The sets of integers, nonnegative integers, nonnegative real numbers and  positive real numbers are denoted by $\mathbb{Z}$, $\mathbb{Z}_{\ge0}$, $\mathbb{R}_{\ge0}$ and $\mathbb{R}_{>0}$, respectively. The $d$-dimensional vector of ones is denoted by $\mathds{1}_d$ (the index will be dropped when the dimension is obvious). For vectors and matrices, the inequality signs $\le$ and $<$ act componentwise. Finally, the vector or matrix obtained by stacking the elements $x_1,\ldots,x_d$ is denoted by $\col_{i=1}^d(x_i)$ or $\col(x_1,\ldots,x_d)$. The diagonal operator $\diag(\cdot)$ is defined analogously. The spectral radius of a matrix $M\in\mathbb{R}^{n\times n}$ is defined as $\varrho(M)=\max\{|\lambda|:\det(\lambda I-M)=0\}$.

\section{Preliminaries on reaction networks}\label{sec:prel}

\subsection{Reaction networks}\label{sec:RN}

We consider here a reaction network with $d$ molecular species $\X{1},\ldots,\X{d}$ that interacts through $K$ reaction channels $\mathcal{R}_1,\ldots,\mathcal{R}_K$ defined as
\begin{equation}
 \mathcal{R}_k:\ \sum_{i=1}^d\zeta_{k,i}^l\X{i}\rarrow{\rho_k}  \sum_{i=1}^d\zeta_{k,i}^r\X{i},\ k=1,\ldots,K
\end{equation}
where $\rho_k\in\mathbb{R}_{>0}$ is the reaction rate parameter and $\zeta_{k,i}^l,\zeta_{k,i}^r\in\mathbb{Z}_{\ge0}$. Each reaction is additionally described by a stoichiometric vector and a propensity function. The stoichiometric vector of reaction  $\mathcal{R}_k$ is given by $\zeta_k:=\zeta_k^r-\zeta_k^l\in\mathbb{Z}^d$ where $\zeta_k^r=\col(\zeta_{k,1}^r,\ldots,\zeta_{k,d}^r)$ and $\zeta_k^l=\col(\zeta_{k,1}^l,\ldots,\zeta_{k,d}^l)$. In this regard, when the reaction $\mathcal{R}_k$ fires, the state jumps from $x$ to $x+\zeta_k$. We define the stoichiometry matrix $S\in\mathbb{Z}^{d\times K}$ as $S:=\begin{bmatrix}
  \zeta_1\ldots\zeta_K
\end{bmatrix}$. When  the kinetics is  mass-action, the propensity function of reaction $\mathcal{R}_k$ is given by  $\textstyle\lambda_k(x)=\rho_k\prod_{i=1}^d\frac{x_i!}{(x_i-\zeta_{k,i}^l)!}$ and is such that  $\lambda_k(x)=0$ if $x\in\mathbb{Z}_{\ge0}^d$ and $x+\zeta_k\notin\mathbb{Z}_{\ge0}^d$. We denote this reaction network by $(\Xz,\mathcal{R})$. Under the well-mixed assumption, this network can be described by a continuous-time Markov process $(X_1(t),\ldots,X_d(t))_{t\ge0}$ with state-space $\mathbb{Z}_{\ge0}^d$; see e.g. \cite{Anderson:11}.

\subsection{Ergodicity of unimolecular and bimolecular reaction networks}

Let us assume here that the network  $(\Xz,\mathcal{R})$ is at most bimolecular and that the reaction rates are all independent of each other. In such a case, the propensity functions are polynomials of at most degree 2 and we can write the propensity vector as
\begin{equation}
  \lambda(x)=\begin{bmatrix}
w_0(\rho_0)\\
W(\rho_u)x\\
Y(\rho_b,x)
  \end{bmatrix}
\end{equation}
where $w_0(\rho_0)\in\mathbb{R}^{n_0}_{\ge0}$, $W(\rho_u)x\in\mathbb{R}^{n_u}_{\ge0}$ and $Y(\rho_b,x)\in\mathbb{R}^{n_b}_{\ge0}$ are the propensity vectors associated the zeroth-, first- and second-order reactions, respectively. Their respective rate parameters are also given by $\rho_0$, $\rho_u$ and $\rho_b$, and according to this structure, the stoichiometric matrix is decomposed as
 $S=:\begin{bmatrix}S_0 &
  S_u & S_b
\end{bmatrix}$. Before stating the main results of the section, we need to introduce the following terminology:
\begin{define}\label{def:Ab}
  The \emph{characteristic matrix} $A(\rho_u)$ and the \emph{offset vector} $b_0(\rho)$ of a bimolecular reaction network $(\Xz,\mathcal{R})$ are defined as
  \begin{equation}
    A(\rho_u):=S_uW(\rho_u)\ \textnormal{and}\ b_0(\rho_0):=S_0w_0(\rho_0).
  \end{equation}
\end{define}
A particularity is that the matrix $A(\rho_u)$ is Metzler (i.e. all the off-diagonal elements are nonnegative) for all $\rho_u\ge0$. This property plays an essential role in the derivation of the results of \cite{Briat:15e} and will also be essential for the derivation of the main results of this paper. It is also important to define the property of ergodicity:
\begin{define}[\cite{Meyn:09}]
  The Markov process associated with the reaction network $(\Xz,\mathcal{R})$ is said to be ergodic if its probability distribution globally converges to a unique stationary distribution. It is exponentially ergodic if the convergence to the unique stationary distribution is exponential.
\end{define}
We then have the following result:
\begin{theorem}[\cite{Briat:13i}]
  Let us consider an irreducible\footnote{Computationally tractable conditions for checking the irreducibility of reaction networks are provided in \cite{Gupta:13}.} bimolecular reaction network $(\Xz,\mathcal{R})$ with fixed rate parameters; i.e. $A=A(\rho_u)$ and $b_0=b_0(\rho_0)$. Assume that there exists a vector $v\in\mathbb{R}^d_{>0}$ such that $v^TS_b=0$ and $v^TA<0$. Then, the  reaction network $(\Xz,\mathcal{R})$ is exponentially ergodic and all the moments are bounded and converging.
\end{theorem}
We also have the following immediate corollary pertaining on unimolecular reaction networks:
\begin{corollary}
  Let us consider an irreducible unimolecular reaction network $(\Xz,\mathcal{R})$ with fixed rate parameters; i.e. $A=A(\rho_u)$ and $b_0=b_0(\rho_0)$. Assume that there exists a vector $v\in\mathbb{R}^d_{>0}$ such that $v^TA<0$. Then, the  reaction network $(\Xz,\mathcal{R})$ is exponentially ergodic and all the moments are bounded and converging.
\end{corollary}

\subsection{Antithetic integral control of unimolecular networks}

Antithetic integral control has been first proposed in \cite{Briat:15e} for solving the perfect adaptation problem in stochastic reaction networks. The underlying idea is to augment the open-loop network $(\Xz,\mathcal{R})$ with an additional set of species and reactions (the controller). The usual set-up is that this controller network acts on the production rate of the molecular species $\X{1}$ (the \emph{actuated species}) in order to steer the mean value of the \emph{controlled species} $\X{\ell}$, $\ell\in\{1,\ldots,d\}$, to a desired set-point (the reference). To the regulation problem, it is often sought to have a controller that can ensure perfect adaptation for the controlled species. As proved in \cite{Briat:15e}, the antithetic integral control motif $(\Zz,\mathcal{R}^c)$ defined with
\begin{equation}
  \phib\rarrow{\mu}\Z{1},\ \phib\rarrow{\theta X_\ell}\Z{2}, \Z{1}+\Z{2}\rarrow{\eta}\phib, \phib\rarrow{kZ_1}\X{1}
\end{equation}
solves this control problem with the set-point being equal to $\mu/\theta$. Above, $\Z{1}$ and $\Z{2}$ are the \emph{controller species}. The four controller parameters $\mu,\theta,\eta,k>0$ are assumed to be freely assignable to any desired value. The first reaction is the \emph{reference reaction} as it encodes part of the reference value $\mu/\theta$ as its own rate. The second one is the \emph{measurement reaction} that produces the species $\Z{2}$ at a rate proportional to the current population of the controlled species $\X{\ell}$. The third reaction is the \emph{comparison reaction} as it compares the populations of the controller species and annihilates one molecule of each when these populations are both positive. Finally, the fourth reaction is the \emph{actuation reaction} that produces the actuated species $\X{1}$ at a rate proportional to the controller species $\Z{1}$.

The following fundamental result states conditions under which a unimolecular reaction network can be controlled using an antithetic integral controller:
\begin{theorem}[\cite{Briat:15e}]\label{th:affine:nominal}
Suppose that the open-loop reaction network $(\Xz,\mathcal{R})$ is unimolecular and that the state-space of the closed-loop reaction network $((\Xz,\Zz),\mathcal{R}\cup\mathcal{R}^c)$ is irreducible. Let us assume that $\rho_0$ and $\rho_u$ are fixed and known (i.e. $A=A(\rho_u)$ and $b_0=b_0(\rho_0)$) and assume, further, that there exist vectors $v\in\mathbb{R}_{>0}^d$, $w\in\mathbb{R}_{\ge0}^d$, $w_1>0$, such that
\begin{equation}
v^TA<0,\ w^TA+e_\ell^T=0\ \textnormal{and}\ \dfrac{\mu}{\theta}>\dfrac{v^Tb_0}{c e_\ell^Tv}
\end{equation}
 where $c>0$ verifies $v^T(A+cI)\le0$.

Then, for any values for the controller rate parameters $\eta,k>0$, (i) the closed-loop network is ergodic, (ii) $\E[X_\ell(t)]\to\mu/\theta$ as $t\to\infty$ and (iii) $\E[X(t)X(t)^T]$ is bounded over time.
\end{theorem}

We can see that the conditions above consist of the combination of an ergodicity condition (i.e. $v^TA<0$) and an output controllability condition for Hurwitz stable matrices $A$ (i.e. $w^TA+e_\ell^T=0$ with $w_1>0$), which are fully consistent with the considered control problem. Note, however, that unlike in the deterministic case, the above result proves that the closed-loop network cannot be unstable if the conditions on the open-loop network are met; i.e. have trajectories that grow unboundedly with time. This is illustrated in more details in the supplemental material of \cite{Briat:15e}.

As the robust/structural output controllability problem has been completely solved in \cite{Briat:16cdc}, we will only focus on checking the ergodicity condition in the rest of the paper.

\section{Robust ergodicity of reaction networks}\label{sec:rob_ergdocity}

\subsection{Preliminaries}

The following lemma will be useful in proving the main results of this section:
\begin{lemma}\label{lem:det}
Let us consider a parameter-dependent Metzler matrix $M(\theta)\in\mathbb{R}^{d\times d}$,  $\theta\in\Theta\subset\mathbb{R}^N_{\ge0}$, where $\Theta$ is compact and connected.  Then, the following statements are equivalent:
\begin{enumerate}[(a)]
  \item The matrix $M(\theta)$ is Hurwitz stable for all $\theta\in\Theta$.
  \item The coefficients of the characteristic polynomial of $M(\theta)$ are positive $\Theta$.
  \item The following conditions hold:
  \begin{enumerate}[(c1)]
    \item there exists a $\theta^*\in\Theta$ such that $M(\theta^*)$ is Hurwitz stable, and
    \item for all $\theta\in\Theta$ we have that $(-1)^d\det(M(\theta))>0$.
  \end{enumerate}
\end{enumerate}
\end{lemma}
\begin{proof}
The proof of the equivalence between (a) and (b) follows, for instance, from \cite{Mitkowski:00} and is omitted. It is also immediate to prove that (b) implies (c) since if  $M(\theta)$ is Hurwitz stable for all $\theta\in\Theta$ then (c1) holds and the constant term of the characteristic polynomial of $M(\theta)$ is positive on $\theta\in\Theta$. Using now the fact that that constant term is equal to $(-1)^d\det(M(\theta))$ yields the result.

To prove that (c) implies (a), we use the contraposition. Hence, let us assume that there exists at least a $\theta_u\in\Theta$ for which the matrix $M(\theta_u)$ is not Hurwitz stable. If such a $\theta_u$ can be arbitrarily chosen in $\Theta$, then this implies the negation of statement (c1) (i.e. for all $\theta^*\in\Theta$ the matrix $M(\theta^*)$ is not Hurwitz stable) and the first part of the implication is proved.

Let us consider now the case where there exists some $\theta_s\in\Theta$ such that $M(\theta_s)$ is Hurwitz stable. Let us then choose a $\theta_u$ and a $\theta_s$ such that $M(\theta_u)$ is not Hurwitz stable and $M(\theta_s)$ is. Since $\Theta$ is connected, then there exists a path $\mathscr{P}\subset\Theta$ from $\theta_s$ and $\theta_u$. From Perron-Frobenius theorem, the dominant eigenvalue, denoted by $\lambda_{PF}(\cdot)$, is real and hence, we have that $\lambda_{PF}(M(\theta_s))<0$ and $\lambda_{PF}(M(\theta_u))\ge0$. Hence, from the continuity of eigenvalues then there exists a $\theta_c\in\mathscr{P}$ such that $\lambda_{PF}(M(\theta_c))=0$, which then implies that $(-1)^d\det(M(\theta_c))=0$, or equivalently, that the negation of (c2) holds. This concludes the proof.
\end{proof}

Before stating the next main result of this section, let us assume that $S_u$ in Definition \ref{def:Ab} has the following form
\begin{equation}
  S_u=\begin{bmatrix}
  S_{dg} & S_{ct} & S_{cv}
  \end{bmatrix}
\end{equation}
where $S_{dg}\in\mathbb{R}^{d\times n_{dg}}$ is a matrix with nonpositive columns, $S_{ct}\in\mathbb{R}^{d\times n_{ct}}$ is a matrix with nonnegative columns and $S_{cv}\in\mathbb{R}^{d\times n_{cv}}$ is a matrix with columns containing exactly one negative entry and at least one positive entry. Also, decompose accordingly $\rho_u$ as $\rho_u=:\col(\rho_{dg},\rho_{ct},\rho_{cv}\}$ and define

\begin{equation*}
  \rho_{\bullet}\in\mathcal{P}_{\bullet}:=[\rho_{\bullet}^-,\rho_{\bullet}^+],\ 0\le\rho_{\bullet}^-\le\rho_{\bullet}^+<\infty
\end{equation*}
where $\bullet\in\{dg,ct,cv\}$ and let $\mathcal{P}_u:=\mathcal{P}_{dg}\times\mathcal{P}_{ct}\times\mathcal{P}_{cv}$.

In this regard, we can alternatively rewrite the matrix $A(\rho_u)$ as $A(\rho_{dg},\rho_{ct},\rho_{cv})$. We then have the following result:
\begin{lemma}\label{lem:A+}
The following statements are equivalent:
  \begin{enumerate}[(a)]
       \item The matrix $A(\rho_u)$ is Hurwitz stable for all $\rho_u\in\mathcal{P}_u$.
    \item The matrix
    \begin{equation}
          A^+(\rho_{cv}):= A(\rho_{dg}^-,\rho_{ct}^+,\rho_{cv})
    \end{equation}
     is Hurwitz stable for all $\rho_{cv}\in\mathcal{P}_{cv}$.
  \end{enumerate}
\end{lemma}
\begin{proof}
The proof that (a) implies (b) is immediate. To prove that (b) implies (a), first note that we have
\begin{equation}
  A(\rho_{dg},\rho_{ct},\rho_{cv})\le A^+(\rho_{cv})=A(\rho_{dg}^-,\rho_{ct}^+,\rho_{cv})
\end{equation}
since for all $(\rho_{dg},\rho_{ct},\rho_{cv})\in\mathcal{P}_{u}$. Using the fact that for two Metzler matrices $B_1,B_2$, the inequality $B_1\le  B_2$ implies  $\lambda_{PF}(B_1)\le \lambda_{PF}(B_2)$ \cite{Berman:94}, then we can conclude that $A(\rho_{dg}^-,\rho_{ct}^+,\rho_{cv})$ is Hurwitz stable for all  $\rho_{cv}\in\mathcal{P}_{cv}$ if and only if  the matrix $A(\rho_{dg},\rho_{ct},\rho_{cv})$ is Hurwitz stable for all $(\rho_{dg},\rho_{ct},\rho_{cv})\in\mathcal{P}_{u}$. This completes the proof.
\end{proof}

\subsection{Unimolecular networks}

The following theorem states the main result on the robust ergodicity of unimolecular reaction networks:
\begin{theorem}\label{th:uni_rob}
 Let $A(\rho_u)\in\mathbb{R}^{d\times d}$ be the characteristic matrix of some unimolecular network and $\rho_u\in\mathcal{P}_u$. Then, the following statements are equivalent:
 \begin{enumerate}[(a)]
   \item The matrix $A(\rho_u)$ is Hurwitz stable for all $\rho_u\in\mathcal{P}_u$.
    \item The matrix
    \begin{equation}
          A^+(\rho_{cv}):= A(\rho_{dg}^-,\rho_{ct}^+,\rho_{cv})
    \end{equation}
     is Hurwitz stable for all $\rho_{cv}\in\mathcal{P}_{cv}$.
     \item There exists a $\rho_{cv}^s\in\mathcal{P}_{cv}$ such that the matrix $A^+(\rho_{cv}^s)$ is Hurwitz stable and the polynomial

     $(-1)^d\det(A^+(\rho_{cv}))$ is positive for all $\rho_{cv}\in\mathcal{P}_{cv}$.
 %  %
    \item There exists a polynomial vector-valued function ${v:\mathcal{P}_{cv}\mapsto\mathbb{R}_{>0}^d}$ of degree at most $d-1$ such that

    ${v(\rho_{cv})^TA^+(\rho_{cv})<0}$ for all $\rho_{cv}\in\mathcal{P}_{cv}$.
 \end{enumerate}
\end{theorem}
\begin{proof}
The equivalence between the statement (a), (b) and (c) directly follows from Lemma \ref{lem:det} and Lemma \ref{lem:A+}. To prove the equivalence between the statements (b) and (d), first remark that (b) is equivalent to the fact that for any $q(\rho_{cv})>0$ on $\mathcal{P}_{cv}$, there exists a unique parameterized vector $v(\rho_{cv})\in\mathbb{R}^d$ such that
$v(\rho_{cv})>0$ and $v(\rho_{cv})^TA^+(\rho_{cv})=-q(\rho_{cv})^T$ for all $\rho_{cv}\in\mathcal{P}_{cv}$. Choosing $q(\rho_{cv})=(-1)^d\mathds{1}_n\det(A^+(\rho_{cv}))$, we get that such a $v(\rho_{cv})$ is given by
\begin{equation}
\begin{array}{rcl}
    v(\rho_{cv})^T      &=&        -(-1)^d\mathds{1}_d^T\det(A^+(\rho_{cv}))A^+(\rho_{cv})^{-1}\\
                                &=&       (-1)^{d+1}\mathds{1}^T_d\textnormal{Adj}(A^+(\rho_{cv}))>0
\end{array}
\end{equation}
for all $\rho_{cv}\in\mathcal{P}_{cv}$. Since the matrix $A^+(\rho_{cv})$ is affine in $\rho_{cv}$, then the adjugate matrix $\textnormal{Adj}(A^+(\rho_{cv})$ contains entries of at most degree $d-1$ and the conclusion follows.
\end{proof}

Checking the condition (c) amounts to solving two problems. The first one is is concerned with the construction of a stabilizer $\rho_{cv}\in\mathcal{P}_{cv}$ for the matrix $A^+(\rho_{cv})$ whereas the second one is about checking whether a polynomial is positive on a compact set. The first problem can be easily solved by checking whether $A^+(\rho_{cv})$ is Hurwitz stable for some randomly chosen point in $\mathcal{P}_{cv}$. For the second one, optimization-based methods can be used such as those based on the Handelman's Theorem combined with linear programming \cite{Handelman:88,Briat:11h} or Putinar's Positivstellensatz combined with semidefinite programming \cite{Putinar:93,Parrilo:00}. Note also that the degree $d-1$ is a worst case degree and that, in fact, polynomials of lower degree will in general be enough for proving the Hurwitz stability of the matrix $A^+(\rho_{cv})$ for all $\rho_{cv}\in\mathcal{P}_{cv}$. For instance, the matrices $A(\rho)$ and $A^+(\rho_{cv})$ are very sparse in general due to the particular structure of biochemical reaction networks. The sparsity property is not considered here but could be exploited to refine the necessary degree for the polynomial vector $v(\rho_{cv})$.

In is important to stress here that Theorem \ref{th:uni_rob} can only be considered when the rate parameters are time-invariant (i.e. constant deterministic or drawn from a distribution). When they are time-varying (e.g. time-varying stationary random variables), a possible workaround relies on the use of a constant vector $v$ as formulated below:
\begin{proposition}[Constant $v$]\label{prop:const_V_Uni}
 Let $A(\rho_u)\in\mathbb{R}^{d\times d}$ be the characteristic matrix of some unimolecular network and $\rho_u\in\mathcal{P}_u$. Then, the following statements are equivalent:
\begin{enumerate}[(a)]
  \item There exists a vector $v\in\mathbb{R}^d_{>0}$ such that ${v^TA^+(\rho_{cv})<0}$ holds for all $\rho_{cv}\in\mathcal{P}_{cv}$.
  \item There exists a vector $v\in\mathbb{R}^d_{>0}$ such that $v^TA^+(\theta)<0$ holds for all $\theta\in\textnormal{vert}(\mathcal{P}_{cv})$ where $\textnormal{vert}(\mathcal{P}_{cv})$ denotes the set of vertices of the set $\mathcal{P}_{cv}$.
\end{enumerate}
\end{proposition}
\begin{proof}
  The proof exploits the affine, hence convex, structure of the matrix $A^+(\rho_{cv})$. Using this property, it is indeed immediate to show that the inequality ${v^TA^+(\rho_{cv})<0}$ holds for all $\rho_{cv}\in\mathcal{P}_{cv}$ if and only if $v^TA^+(\theta)<0$ holds for all $\theta\in\textnormal{vert}(\mathcal{P}_{cv})$ (see e.g. \cite{Briat:book1} for a similar arguments in the context of quadratic Lyapunov functions).
\end{proof}

The above result is connected to the existence of a linear copositive Lyapunov function for a linear positive switched system with matrices in the family $\{A(\theta):\theta\in\textnormal{vert}(\mathcal{P}_{cv})\}$ for which many characterizations exist; see e.g. \cite{Fornasini:10,Mason:07}.

\subsection{Bimolecular networks}

In the case of bimolecular networks, we have the following result:
\begin{proposition}
 Let $A(\rho_u)\in\mathbb{R}^{d\times d}$ be the characteristic matrix of some bimolecular network and $\rho_u\in\mathcal{P}_u$. Then, the following statements are equivalent:
 \begin{enumerate}[(a)]
 %  \item The matrix $A(\rho)$ is Hurwitz stable for all $\rho\in\mathcal{P}$.
   %
\item There exists a polynomial vector-valued function ${v:\mathcal{P}_{u}\mapsto\mathbb{R}_{>0}^d}$ such that
   \begin{equation}\label{eq:dmlsqdksdsmdkml}
     v(\rho_u)>0,\ v(\rho_u)^TS_b=0\ \textnormal{and}\ v(\rho_u)^TA(\rho_u)<0
   \end{equation}
for all $\rho_u\in\mathcal{P}_u$.
 %  \item All the coefficients of the characteristic polynomial $P(\lambda):=\det(\lambda I-A(\rho))$ are positive for all $\rho\in\mathcal{P}$.
   %
%   \item There exists an $\alpha>0$ such that $p_i(\rho)\ge\alpha$ for all $\rho\in\mathcal{P}$.
%   %
%\item There exists a polynomial vector-valued function ${\tilde{v}:\mathcal{P}_{cv}\mapsto\mathbb{R}_{>0}^d}$ of degree at most $d-1$ such that
%    %
%   \begin{equation}
%     \tilde{v}(\rho_{cv})>0,\quad \tilde{v}(\rho_{cv})^TS_b=0\quad \textnormal{and}\quad \tilde{v}(\rho_{cv})^TA^+(\rho_{cv})<0
%   \end{equation}
%for all $\rho_{cv}\in\mathcal{P}_{cv}$.
%
%   \item There exist an $L\in\mathbb{N}$ and an $\alpha>0$ such that $v(\rho)^TS_b=0$ and
%   \begin{itemize}
%     \item $v_i(\rho)\in\mathcal{S}_L(\pi^-,\pi^+)$ for all $i=1,\ldots,d$.
%     \item $-[v(\rho)A(\rho)]e_i\in\mathcal{S}_L(\pi^-,\pi^+)$ for all $i=1,\ldots,d$.
%   \end{itemize}
%   %
%   \item There exist an $L\in\mathbb{N}$, an $\alpha>0$ and some $\tau^i_\sigma\ge0$, $0\le|\sigma|\le L$, such that $\kappa^i_{l}(\tau^i_\sigma)\ge0$ for all $0\le|\eta|\le\bar{\eta}_i:=\max\{\textnormal{degree}(p_i(\rho)),L\}$ where
%   \begin{equation}
%     \sum_{0\le|\eta|\le\bar{\eta}_i}\kappa_{\eta}^i(\tau^i_\sigma)\rho^\eta=p_i(\rho)-s_i(\rho)+\alpha
%   \end{equation}
%   and $s_i(\rho)\in\mathcal{S}_L(\pi^-,\pi^+)$.
\item There exists a polynomial vector-valued function ${\tilde{v}:\mathcal{P}_{cv}\mapsto\mathbb{R}^{d-n_b}}$ such that
     \begin{equation}
     \tilde{v}(\rho_{cv})^TS_b^\bot>0\ \textnormal{and}\ \tilde{v}(\rho_{cv})^TS_b^\bot A^+(\rho_{cv})<0
   \end{equation}
for all $\rho_{cv}\in\mathcal{P}_{cv}$ and where $n_b:=\rank(S_b)$ and $S_b^\bot S_b=0$, $S_b^\bot$ full-rank.
 \end{enumerate}
\end{proposition}
\begin{proof}
  It is immediate to see that (a) implies (b). To prove the converse, first note that we have that $v(\rho_{cv})=(S_b^\bot)^T\tilde{v}(\rho_{cv})$ verifies $v(\rho_{cv})^TS_b=0$ and $v(\rho_{cv})>0$ for all $\rho_{cv}\in\mathcal{P}_{cv}$. This proves the equality and the first inequality in \eqref{eq:dmlsqdksdsmdkml}. Observe now that for any $\rho_u\in\mathcal{P}_u$, there exists a nonnegative matrix $\Delta(\rho_{dg},\rho_{ct})\in\mathbb{R}^{d\times d}_{\ge0}$ such that $A(\rho_u)=A^+(\rho_{cv})-\Delta(\rho_{dg},\rho_{ct})$. Hence, we have that
  \begin{equation}
  \begin{array}{rcl}
        v(\rho_{cv})^TA(\rho_u)&=&v(\rho_{cv})^T(A^+(\rho_{cv})-\Delta(\rho_{dg},\rho_{ct}))\\
        &\le&v(\rho_{cv})^TA^+(\rho_{cv})<0
  \end{array}
  \end{equation}
  which proves the result.
\end{proof}

As in the unimolecular case, we have been able to reduce the number of parameters by using an upper-bound on the characteristic matrix. It is also interesting to note that the condition $\tilde{v}(\rho_{cv})^TS_b^\bot A^+(\rho_{cv})<0$ can be sometimes brought back to a problem of the form $\tilde{v}(\rho_{cv})^TM(\rho_{cv})<0$ for some square, and often Metzler, matrix $M(\rho_{cv})$ which can be dealt in the same way as in the unimolecular case.

%\red{EXPLAIN: The second is interesting as it shows that we can consider an upper-bound on the matrix and therefore reduce the number of parameters. Additionally, we can get a projected condition in terms of the left null-space of $S_b$ that can sometimes be brought back to a Hurwitz stability condition that can be dealt in the same way as for unimolecular networks.}
The following result can be used when the parameters are time-varying and is the bimolecular analogue of Proposition \ref{prop:const_V_Uni}:
\begin{proposition}[Constant $v$]\label{prop:const_V_Bi}
 Let $A(\rho_u)\in\mathbb{R}^{d\times d}$ be the characteristic matrix of some bimolecular network and $\rho_u\in\mathcal{P}_u$. Then, the following statements are equivalent:
\begin{enumerate}[(a)]
  \item There exists a vector $v\in\mathbb{R}^d_{>0}$ such that $v^TS_b=0$ and $v^TA^+(\rho_{cv})<0$ hold for all $\rho_{cv}\in\mathcal{P}_{cv}$.
  \item There exists a vector $v\in\mathbb{R}^d_{>0}$ such that $v^TS_b=0$ and $v^TA^+(\theta)<0$ hold for all $\theta\in\textnormal{vert}(\mathcal{P}_{cv})$.
\end{enumerate}
\end{proposition}

\section{Structural ergodicity of unimolecular reaction networks}\label{sec:struct_ergdocity}

We are interested in this section in the structural stability of the characteristic matrix of given unimolecular network. Hence, we have in this case $\mathcal{P}_\bullet:=\mathbb{R}^{n_\bullet}_{>0}$, $\bullet\in\{u,dg,ct,cv\}$ where $n_\bullet$ is the dimension of the vector $\rho_{\bullet}$.

\subsection{A preliminary result}\label{lem:rho1}

\begin{lemma}\label{lem:rescaling}
 Let $A(\rho_u)\in\mathbb{R}^{d\times d}$ be the characteristic matrix of some unimolecular network and $\rho_u\in\mathcal{P}_u$. Then, the following statements are equivalent:
  \begin{enumerate}[(a)]
    \item For all $\rho_{dg}\in\mathcal{P}_{dg}$ and a $\rho_{cv}\in\mathcal{P}_{cv}$, the matrix $A(\rho_{dg},\rho_{cv},0)$ is Hurwitz stable.
    \item The matrix $A(\mathds{1},\rho_{cv},0)$ is Hurwitz stable for all $\rho_{cv}\in\mathcal{P}_{cv}$.
  \end{enumerate}
\end{lemma}
\begin{proof}
  The proof that (a) implies (b) is immediate. To prove the reverse implication, we use contraposition and we assume that there exist a $\rho_{dg}\in\mathcal{P}_{dg}$ and a $\rho_{cv}\in\mathcal{P}_{cv}$ such that $A(\rho_{dg},\rho_{cv},0)$ is not Hurwitz stable. Then, we clearly have that
  \begin{equation}
    A(\rho_{dg},\rho_{cv},0)\le A(\theta \mathds{1},\rho_{cv},0)
  \end{equation}
  where $\theta=\min(\rho_{dg})$ and hence $A(\theta \mathds{1},\rho_{cv},0)$ is not Hurwitz stable. Since $A(\theta \mathds{1},\rho_{cv},0)$ is affine in $\theta$ and $\rho_{cv}$, then we have that  $\theta A(\mathds{1},\rho_{cv}/\theta,0)$ and since $\theta$ is independent of $\rho_{cv}$, then we get that the matrix $A(\mathds{1},\tilde{\rho}_{cv},0)$ is not Hurwitz stable for some $\tilde{\rho}_{cv}\in\mathcal{P}_{cv}$. The proof is complete.
\end{proof}

\subsection{Main result}

\begin{theorem}\label{th:structural}
Let $A(\rho_u)\in\mathbb{R}^{d\times d}$ be the characteristic matrix of some unimolecular network and $\rho_u\in\mathbb{R}^{n_u}_{>0}$. Then, the following statements are equivalent:
\begin{enumerate}[(a)]
\item The matrix $A(\rho_u)$ is Hurwitz stable for all $\rho_u\in\mathbb{R}^{n_u}_{>0}$.
%  \item The matrix $A_0:=A(\mathds{1},\mathds{1},0)$ is Hurwitz stable and $\textnormal{trace}(W_{ct}A_0^{-1}S_{ct})\ne0$.
%
\item There exists a polynomial vector $v(\rho_u)\in\mathbb{R}^d$ of degree at most $d-1$ such that $v(\rho_u)>0$ and $v(\rho_u)^TA(\rho_u)<0$ for all $\rho_u\in\mathbb{R}^{n_u}_{>0}$.
%  \item The matrix $A_0:=A(\mathds{1},\mathds{1},0$ is Hurwitz stable and for all $\rho_{ct}\in\mathbb{P}_{ct}$ we have that $\det(A_0+S_{ct}D(\rho_{ct})W_{ct})>0$.
%  %
  \item There exists a $\rho_u^s\in\mathbb{R}^{n_u}_{>0}$ such that the matrix $A(\rho_u^s)$ is Hurwitz stable and the polynomial $(-1)^d\det(A(\rho_u))$ is positive for all $\rho_u\in\mathbb{R}^{n_u}_{>0}$.
%
%  \item There exist a $\rho_{dg}^s\in\mathbb{R}^{n_{dg}}_{>0}$ and a $\rho_{cv}^s\in\mathbb{R}^{n_{cv}}_{>0}$ such that $A_\rho:=A(\rho_{dg}^s,\rho_{cv}^s,0)$ is Hurwitz stable and for all $\rho_{dg}\in\mathbb{R}^{n_{dg}}_{>0}$ and $\rho_{cv}\in\mathbb{R}^{n_{cv}}_{>0}$, the eigenvalues of the nonnegative matrix $W_{ct}A_\rho^{-1}S_{ct}$ are all equal to 0.
\item For all $\rho_{dg}\in\mathbb{R}^{n_{dg}}_{>0}$ and a $\rho_{cv}\in\mathbb{R}^{n_{cv}}_{>0}$, the matrix $A_\rho:=A(\rho_{dg},\rho_{cv},0)$ is Hurwitz stable and we have that $\varrho(W_{ct}A_\rho^{-1}S_{ct})=0$.
  %
%  \item The matrix $A_0:=A(\mathds{1},\mathds{1},0$ is Hurwitz stable and for all $\rho_{ct}\in\mathbb{P}_{ct}$ we have that $\det(A_0+S_{ct}D(\rho_{ct})W_{ct})>0$.
  %
    \item The matrix $A_n(\rho_{cv}):=A(\mathds{1},\rho_{cv},0)$ is Hurwitz stable for all $\rho_{cv}\in\mathbb{R}^{n_{cv}}_{>0}$ and $\varrho(W_{ct}A_n(\rho_{cv})^{-1}S_{ct})=0$ for all $\rho_{cv}\in\mathbb{R}^{n_{cv}}_{>0}$.
  %
%  \item \red{The matrix $A_{\mathds{1}}:=A(\mathds{1},\mathds{1},0)$ is Hurwitz stable and for all $j=1,\ldots,n_{cv}$, there exists a $k_i\in\mathcal{I}_j$ such that there is a path from node $k_i$ to node $k_o$ for some $k_o\in\mathcal{O}_j$ in the directed graph $(\mathcal{V},\mathcal{E})$ defined with $\mathcal{V}:=\{1,\ldots,d\}$ and
%\begin{equation*}
%    \mathcal{E}:=\{(m,n):\ e_{n}^TA_{\mathds{1}}e_{m}\ne0,\ m,n\in V,\ m\ne n\}
%\end{equation*}
%and where
%\begin{equation}
%  \begin{array}{rcl}
%    \mathcal{I}_j&=&\{i:[S_{ct}]_{ij}\ne0,i=1,\ldots,d\}\\
%    \mathcal{O}_j&=&\{i:[W_{ct}]_{ji}\ne0,i=1,\ldots,d\}\\
%  \end{array}
%\end{equation}}
\end{enumerate}
Moreover, when each column of $S_{cv}$ contains exactly two nonzero entries, one being equal to $-1$ and one being equal to 1, then the above statements are also equivalent to
\begin{enumerate}[(a)]
\setcounter{enumi}{5}
    \item The matrix $A_{\mathds{1}}:=A(\mathds{1},\mathds{1},0)$ is Hurwitz stable and $\varrho(W_{ct}A_{\mathds{1}}^{-1}S_{ct})=0$.
        \end{enumerate}
\end{theorem}
\begin{proof}
The equivalence between the three first statements has been proved in Theorem \ref{th:uni_rob}. Let us prove now that (c) implies (d). Assuming that (c) holds, we get that the existence of a $\rho_u^s=\col(\rho_{dg}^s,\rho_{cv}^s,\rho_{ct}^s)$ such that the matrix $A(\rho_u^s)$ is Hurwitz stable immediately implies that the matrix $A_\rho=A(\rho_{dg},\rho_{cv},0)$ is Hurwitz stable since we have that $A_\rho\le A(\rho_u)$ and, therefore $\lambda_{PF}(A_\rho)\le \lambda_{PF}(A(\rho_u))<0$.
Using now the determinant formula, we have that
\begin{equation}\label{ea:djskdjl}
    \det(A(\rho_u))=\det(A_\rho)\det(I-D(\rho_{ct})W_{ct}A_\rho^{-1}S_{ct})
\end{equation}
where $D(\rho_{ct}):=\diag(\rho_{ct})$ and $W_{ct}$ is defined such that ${\diag(\rho_{ct})W_{ct}x}$ is the vector of propensity functions associated with the catalytic reactions. Since $A_\rho$ is Hurwitz stable then the determinant has fixed sign and is positive if $d$ is even, negative otherwise.  Hence, this implies that
\begin{equation}
\det(I-D(\rho_{ct})M)>0
\end{equation}
for all $\rho_{ct}\in\mathbb{R}^{n_{ct}}_{>0}$ where $M:=-W_{ct}A_\rho^{-1}S_{ct}$. Since the matrices $W_{ct},S_{ct}$ are nonnegative, the diagonal entries of $D(\rho_{ct})$ are positive and $A_\rho^{-1}$ is nonpositive (since $A_\rho$ is Metzler and Hurwitz stable), then $M$ is nonnegative. Then, by the Perron-Frobenius theorem, we have that $\lambda_{PF}(M)=\varrho(M)$. Clearly, the fact that \eqref{ea:djskdjl} holds for all $\rho_{ct}\in\mathbb{R}^{n_{ct}}_{>0}$ implies that $\lambda_{PF}(M)=\varrho(M)=0$ since, otherwise, there would exist a $D(\rho_{ct})$ such that $\det(I-D(\rho_{ct})M)=0$. This completes the argument.

The converse (i.e. (d) implies (c)) can be proven by noticing that if $A_\rho$ is Hurwitz stable, then $A_\rho+\epsilon S_{ct}W_{ct}$ remains Hurwitz stable for some sufficiently small $\epsilon>0$. This proves the existence of a  $\rho_u^s\in\mathbb{R}^d_{>0}$ such that the matrix $A(\rho_u^s)$.  Using the determinant formula, it is immediate to see that the second statement implies the determinant condition of statement (c).

The equivalence between the statements (d) and (e) comes from Lemma \ref{lem:rescaling} and the fact that the sign-pattern of the inverse of a Hurwitz stable Metzler matrix is uniquely defined by its sign-pattern.

Let us now focus on the equivalence between the statements (d) and (f) under the assumption that each column of $S_{cv}$ contains exactly one entry equal to $-1$ and  one equal to 1. Assume w.l.o.g that $S_{dg}=\col(-I_{n_{dg}},0)$. Then, we have that $\mathds{1}_{d}^TA(\rho_{dg},\rho_{cv},0)=\begin{bmatrix}
  -\rho_{dg}^T & 0
\end{bmatrix}$. Hence, the function $V(z)=\mathds{1}_d^Tz$ is a weak Lyapunov function for the linear positive system $\dot{z}=A(\rho_{dg},\rho_{cv},0)z$. Invoking LaSalle's invariance principle, we get that the matrix is Hurwitz stable if the matrix
\begin{equation}
  A^{22}(\rho_{dg},\rho_{cv}):=\begin{bmatrix}
    0\\ I
  \end{bmatrix}^TA(\rho_{dg},\rho_{cv},0)  \begin{bmatrix}
    0 \\ I
  \end{bmatrix}
\end{equation}
is Hurwitz stable for all $(\rho_{dg},\rho_{cv})\in\mathbb{R}^{n_{dg}}_{>0}\times\mathbb{R}^{n_{cv}}_{>0}$. Note that this is a necessary condition for the matrix $A(\rho_{dg},\rho_{cv},0) $ to be Hurwitz stable for all rate parameters values. Hence, this means that the stability of the matrix $A_\rho$ is equivalent to the Hurwitz stability of $A_{\mathds{1}}:=A(\mathds{1},\mathds{1},0)$. Finally, since $A^{22}(\rho_{dg},\rho_{cv})$ is Hurwitz stable, then we have that $\mathds{1}^TA^{22}(\mathds{1},\mathds{1})<0$. The proof is complete.
\end{proof}

\section{Examples}\label{sec:examples}

\subsection{Example 1 - SIR model}

Let us consider the open stochastic SIR model considered in \cite{Briat:13i} described by the matrices
\begin{equation}
  A = \begin{bmatrix}
    -\gamma_s & 0 & k_{rs}\\
    0 & -(\gamma_i+k_{ir}) & 0\\
    0 & k_{ir} & -(\gamma_r+k_{rs})
  \end{bmatrix},\ S_b=\begin{bmatrix}
-1\\
1\\
0
  \end{bmatrix}
\end{equation}
where all the parameters are positive. The first, second and third states are associated with susceptible, infectious and removed people, respectively. The reaction network consists of 6 reactions: 5 unimolecular reactions including 3 degradation reactions (with rates $\gamma_s,\gamma_i$ and $\gamma_r$) and 2 conversion reactions (with rates $k_{ir}$ and $k_{rs}$), and 1 bimolecular reaction that implements the contamination reaction that converts one susceptible person to an infectious one.

The constraint $v^TS_b=0$ enforces that $v=\tilde{v}^TS_b^\bot$, $\tilde{v}>0$, where $S_b^\bot=\begin{bmatrix}
1 & 1 & 0\\
0 & 0 & 1
\end{bmatrix}$. This leads to
\begin{equation}
  \tilde{v}^TS_b^\bot A<0\Leftrightarrow \tilde{v}^T\begin{bmatrix}
-(\gamma_i+k_{ir}) & k_{rs}\\
 k_{ir} & -(\gamma_r+k_{rs})
  \end{bmatrix}<0.
\end{equation}
Since the entries are not independent, the use of sign-matrices or interval matrices are conservative. However, if we use Theorem \ref{th:structural}, then we can just substitute the parameters by 1 and observe that the resulting matrix is Hurwitz stable to prove the structural stability of the matrix. Alternatively, we can take $\tilde{v}=\mathds{1}$ and obtain %
\begin{equation}
  \tilde{v}^T\begin{bmatrix}
-(\gamma_i+k_{ir}) & k_{rs}\\
 k_{ir} & -(\gamma_r+k_{rs})
  \end{bmatrix}=\begin{bmatrix}
-\gamma_i & -\gamma_r
  \end{bmatrix}<0
\end{equation}
from which the same result follows.

\subsection{Example 2 - Circadian Clock}

We consider the circadian clock-model of \cite{Vilar:02} which is described by the matrices
  \begin{equation}
  A = \begin{bmatrix}
    -\delta_{M_A} & 0 & 0 & 0 & 0\\
    \beta_A & -\delta_A & 0 & 0 & 0\\
    0 & 0 & -\delta_{M_R} & 0 & 0\\
    0 & 0 & \beta_R & -\delta_R & \delta_A\\
    0 & 0 & 0 & 0 & -\delta_A
  \end{bmatrix},\ S_b=\begin{bmatrix}
0\\
-1\\
0\\
-1\\
1
  \end{bmatrix}
\end{equation}
where all the parameters are positive. Even if initially the network consists of 9 species and 16 reactions, the problem can be reduced to the above problem involving 5 species (the activator protein $A$ and its corresponding mRNA $M_A$, the activator protein $R$ and its corresponding mRNA $M_R$ and a dimer $C$ composed of an activator protein $A$ and a repressor protein $R$). There are 8 reactions: 7 unimolecular reactions including 4 degradation reactions  (with rates $\delta_{M_A},\delta_{A},\delta_{M_R}$ and $\delta_{R}$), 2 catalytic reactions (with rates $\beta_{A}$ and $\beta_{R}$) and one conversion reaction with rate $\delta_A$ (with rate identical to the degradation rate of the activator protein $A$), and one bimolecular reaction corresponding to the binding reaction of $A$ and $R$.

As in the previous example, the condition reduces to
\begin{equation}
  \tilde{v}^TS_b^\bot A<0\Leftrightarrow \tilde{v}^T\begin{bmatrix}
-\delta_{M_A} & 0 & 0 & 0\\
\beta_A & -\delta_A & 0 &0\\
0 & 0 & -\delta_{M_R} & 0\\
0 & 0 & \beta_R & -\delta_R
  \end{bmatrix}<0
\end{equation}
where $\tilde{v}>0$. Clearly, we have four degradation reactions and two catalytic ones. Using the last statement of Theorem \ref{th:structural} we get that the matrix $A_{\mathds{1}}=-I$. We also have in this case that
\begin{equation}
  W_{ct}=\begin{bmatrix}
    1 & 0 & 0 & 0\\
    0 & 0 & 1 & 0
  \end{bmatrix}\quad\textnormal{and}\quad S_{ct}=\begin{bmatrix}
    0 & 1 & 0 & 0\\
    0 & 0 & 0 & 1
  \end{bmatrix}^T
\end{equation}
and, hence, $W_{ct}A_{\mathds{1}}^{-1}S_{ct}=0$. Hence, the system is structurally stable. Alternatively, the triangular structure of the matrix would also lead to the same conclusion.

\subsection{Example 3 - Toy model}

Let us consider here the following toy network where
\begin{equation}
  A=\begin{bmatrix}
-(\gamma_1+\alpha_1) & 0 & k_1\\
k_2 & -(\gamma_2+\alpha_2) & 0\\
0 & k_3 & -k_1
  \end{bmatrix}.
\end{equation}
Assume that $\alpha_1=k_2$ and $\alpha_2=k_3$. In such a case, the network consists of 5 unimolecular reactions: 2 degradation reactions with rates $\gamma_1$ and $\gamma_2$ and 3 conversion reactions with rates $k_1$, $k_2$ and $k_3$. Then, we get that
\begin{equation}
  A_{\mathds{1}}=\begin{bmatrix}
-2 & 0 & 1\\
1 & -2 & 0\\
0 & 1 & -1
  \end{bmatrix}
\end{equation}
is Hurwitz stable and hence that the matrix is structurally stable. However, if we assume now that $\alpha_1=\alpha_2=0$, then the network still consists of 5 unimolecular reactions but now we have 2 degradation reactions with rates $\gamma_1$ and $\gamma_2$, 1 conversion reaction with rate $k_1$ and 2 catalytic reactions with rates $k_2$ and $k_3$. In such a case, we get
\begin{equation}
  A_{\mathds{1}}=\begin{bmatrix}
-1 & 0 & 1\\
0 & -1 & 0\\
0 & 0 & -1
  \end{bmatrix}\quad \textnormal{and}\quad A_{\mathds{1}}^{-1}=\begin{bmatrix}
-1 & 0 & -1\\
0 & -1 & 0\\
0 & 0 & -1
  \end{bmatrix}
\end{equation}
where $A_{\mathds{1}}$ is Hurwitz stable. We have in this case that
\begin{equation}
  W_{ct}=\begin{bmatrix}
    1 & 0 & 0\\
    0 &  1 & 0
  \end{bmatrix}\quad\textnormal{and}\quad S_{ct}=\begin{bmatrix}
    0 & 1  & 0\\
    0 & 0 &  1
  \end{bmatrix}^T
\end{equation}
and hence
\begin{equation}
  W_{ct}A_{\mathds{1}}^{-1}S_{ct}=\begin{bmatrix}
    0 & 1\\
    1 & 0
  \end{bmatrix}
\end{equation}
which has a spectral radius equal to 1. Hence, the matrix is not structurally stable. Define now
\begin{equation}
  A^+(k_1)=\begin{bmatrix}
    -\gamma_1^- & 0 & k_1\\
    k_2^+ & -\gamma_2^- & 0\\
    0 & k_3^+ & -k_1
  \end{bmatrix}.
\end{equation}
Using a perturbation argument, we can prove that the 0-eigenvalue of $A^+(0)$ locally bifurcates to the open left half-plane for some sufficiently small $k_1>0$ if and only if $k_2^+k_3^+-\gamma_1^-\gamma_2^-<0$. Hence, there exists a $k_1>0$ such that $A^+(k_1)$ is Hurwitz stable if and only if $k_2^+k_3^+-\gamma_1^-\gamma_2^-<0$. Noting now that
\begin{equation}
  \det(A^+(k_1))=k_1(k_2^+k_3^+-\gamma_1^-\gamma_2^-)<0
\end{equation}
and, hence, the determinant never switches sign, which proves that the matrix $A^+(k_1)$ is structurally stable.

\section{Conclusion}

Several conditions have been proposed for checking the robust and the structural ergodicity of unimolecular and a certain class of bimolecular reaction networks. When certain conditions are met, the conditions can be verified using convex programming tools and/or through the computation of simple algebraic quantities. These results complement those of \cite{Briat:16cdc} by providing a more accurate set of conditions. Yet, improvements are still possible at the level of Theorem \ref{th:uni_rob} and Theorem \ref{th:structural} in order to avoid the use of complex numerical tools or to remove the need for simplifying assumptions.

{%\small
\bibliographystyle{IEEEtran}
%\bibliography{../../../Lastbib/global,../../../Lastbib/briat}
% Generated by IEEEtran.bst, version: 1.14 (2015/08/26)

}

\end{document}